\documentclass[12pt]{amsart}

\usepackage{amscd,amssymb,amsopn,amsmath,amsthm,mathrsfs,graphics,amsfonts,enumerate,verbatim,calc
}

\usepackage[all]{xy}

\usepackage{color}

\usepackage[OT2,OT1]{fontenc}
\newcommand\cyr{%
\renewcommand\rmdefault{wncyr}%
\renewcommand\sfdefault{wncyss}%
\renewcommand\encodingdefault{OT2}%
\normalfont
\selectfont}
\DeclareTextFontCommand{\textcyr}{\cyr}

\usepackage{amssymb,amsmath}

\DeclareFontFamily{OT1}{rsfs}{}
\DeclareFontShape{OT1}{rsfs}{n}{it}{<-> rsfs10}{}
\DeclareMathAlphabet{\mathscr}{OT1}{rsfs}{n}{it}

\numberwithin{equation}{section}
\newtheorem{theorem}{Theorem}[section]
\newtheorem{lemma}[theorem]{Lemma}
\newtheorem{proposition}[theorem]{Proposition}
\newtheorem{corollary}[theorem]{Corollary}

\newtheorem{question}{Question}

\theoremstyle{definition}
\newtheorem{definition}[theorem]{Definition}
\newtheorem{remark}[theorem]{Remark}
\theoremstyle{remark}

\newtheorem{acknowledgement}{Acknowledgement}

\renewcommand{\ker}{\operatorname{Ker}}

\newcommand{\fm}{\frak{m}}
\newcommand{\fp}{\frak{p}}
\newcommand{\fq}{\frak{q}}

\begin{document}
\title[Frobenius test exponents]{Notes on the Frobenius test exponents}

\author[Duong Thi Huong]{Duong Thi Huong}
\address{Department of Mathematics, Thang Long University, Hanoi, Vietnam}
\email{duonghuongtlu@gmail.com}

\author[Pham Hung Quy]{Pham Hung Quy}
\address{Department of Mathematics, FPT University, Hanoi, Vietnam}
\email{quyph@fe.edu.vn}

\thanks{
The second author is partially supported by a fund of Vietnam National Foundation for Science
and Technology Development (NAFOSTED) under grant number
101.04-2017.10.}


\begin{abstract} In this paper we show that the Frobenius test exponent for parameter ideals of a local ring of prime characteristic is always bigger than or equal to its Hartshorne-Speiser-Lyubeznik number. Our argument is based on an isomorphism of Nagel and Schenzel on local cohomology for which we will provide an elementary proof.
\end{abstract}

\maketitle

\section{Introduction}
Let $R$ be a Noetherian commutative ring of prime characteristic $p>0$, and $I$ an ideal of $R$. The {\it Frobenius closure} of $I$ is $I^F = \{x \mid x^{p^e} \in I^{[p^e]} \text{ for some } e \ge 0\}$, where $I^{[p^e]} = (r^{p^e} \mid r \in I)$ is the $e$-th Frobenius power of $I$. It is hard to compute $I^F$. By the Noetherianess of $R$ there is an integer $e$, depending on $I$, such that $(I^F)^{[p^e]} = I^{[p^e]}$. We call the smallest number $e$ satisfying the condition the {\it Frobenius test exponent} of $I$, and denote it by $Fte(I)$. It is natural to expect the existence of a uniform number $e$, depending only on the ring $R$, such that, for every ideal $I$ we have $(I^F)^{[p^e]} = I^{[p^e]}$, i.e. $Fte(I) \le e$ for every ideal $I$. If we have a positive answer to
this question, then the two conditions $x \in I^F$ and $x^{p^e} \in I^{[p^e]}$ are equivalent. This gives in particular a finite test for the Frobenius closure. We call such a number $e$ a {\it Frobenius test exponent} for the ring $R$. However, Brenner \cite{B06} gave two-dimensional normal standard graded domains with no Frobenius test exponent. In contrast, Katzman and Sharp \cite{KS06} showed the existence of a uniform bound of Frobenius test exponents if we restrict to the class of parameter ideals in a Cohen-Macaulay local ring. For any local ring $(R, \fm)$ we define the {\it Frobenius test exponent for parameter ideals}, denoted by $Fte(R)$, is the smallest integer $e$ such that $(\fq^F)^{[p^e]} = \fq^{[p^e]}$ for every parameter ideal $\fq$ of $R$, and $Fte(R) = \infty$ if we have no such integer. Katzman and Sharp asked whether $Fte(R) < \infty$ for any (equidimensional) local ring. Furthermore, the authors of \cite{HKSY06} confirmed the question for generalized Cohen-Macaulay local rings. Recently the second author gave a positive answer for the question for $F$-nilpotent rings \cite{Q18}.\\
The main idea in \cite{HKSY06, KS06} is connecting the Frobenius test exponent for parameter ideals with an invariant defined by the Frobenius actions on the local cohomology modules $H^i_{\fm}(R)$, namely {\it the Hartshorne-Speiser-Lyubeznik number} of $H^i_{\fm}(R)$. Recall that the Frobenius endomorphism $F: R \to R, x \mapsto x^p$ induces Frobenius actions on the local cohomology modules $H^i_{\fm}(R)$ for all $i \ge 0$. Roughly speaking, the Hartshorne-Speiser-Lyubeznik number of $H^i_{\fm}(R)$, denoted by $HSL(H^i_{\fm}(R))$, is a nilpotency index of Frobenius actions on $H^i_{\fm}(R)$ for all $i \ge 0$ (see Section 3 for details).  The Hartshorne-Speiser-Lyubeznik number of $R$ is $HSL(R) = \max \{HSL(H^i_{\fm}(R)) \mid i = 0, \ldots, \dim R\}$. Moreover, the Frobenius action $F: H^i_{\fm}(R) \to H^i_{\fm}(R)$ is injective for all $i \ge 0$ (in this case, $R$ is called {\it $F$-injective}) if and only if $HSL(R) = 0$. If $R$ is Cohen-Macaulay, Katzman and Sharp proved the equality $Fte(R) = HSL(R)$. The main result of the present paper as follows
\begin{theorem}\label{main thm}
Let $(R, \fm)$ be a local ring of positive characteristic $p>0$ of dimension $d$. Then $Fte(R) \ge HSL(R)$.
\end{theorem}
Our main technique is to analyze the local cohomology modules by using the Nagel-Schenzel isomorphism (cf. \cite[Proposition 3.4]{NS95}). Then we can consider any local cohomology $H^i_{\fm}(R)$ as a submodule of a top local cohomology whose Frobenius action can be understood explicitly. In the next section, we will give a new and simple proof for Nagel-Schenzel's isomorphism. The main result will be proved in Section 3 (Theorem \ref{inequality}). We also prove that the Frobenius test exponent for parameter ideals has a good behavior under localization (Proposition \ref{localization}).

\section{Nagel-Schenzel's isomorphism}
In this section, let $R$ be a commutative Noetherian ring, $M$ a finitely generated $R$-module and $I$ an ideal of $R$. The use of $I$-filter regular sequences on $M$ provide an useful technique for the study of local cohomology. In \cite[Proposition 3.4]{NS95} Nagel and Schenzel proved the following useful theorem (see also \cite{AS03}).
\begin{theorem}\label{Nagel-Schenzel}
Let $I$ be an ideal of a Noetherian ring $R$ and $M$ a finitely generated $R$-module. Let $x_1,\ldots,x_t$ an
$I$-filter regular sequence of $M$. Then we have
\[
H^i_{I}(M) \cong
\begin{cases}
H^i_{(x_1,\ldots, x_t)}(M) &\text{ if } i<t\\
H^{i-t}_{I}(H^t_{(x_1, \ldots, x_t)}(M)) &\text{ if } i \ge t.
\end{cases}
\]
\end{theorem}
The most important case of Theorem \ref{Nagel-Schenzel} is $i = t$, and so $H^t_{I}(M) \cong H^0_{I}(H^t_{(x_1, \ldots, x_t)}(M))$ a submodule of $H^t_{(x_1, \ldots, x_t)}(M)$. Recently, many applications of this fact have been found \cite{DQ18, PQ18, QS17}. It should be noted that Nagel-Schenzel's theorem was proved by using spectral sequences. The aim of this section is to give an elementary proof for Theorem \ref{Nagel-Schenzel} based on standard arguments of local cohomology \cite{BS98}.
We recall the definition and some simple properties of $I$-filter regular sequences.
\begin{definition}
Let $M$ be a finitely generated module $R$ and let $x_1,\ldots,x_t \in I$ be a sequence of elements of $R$. Then we say that $x_1,\ldots,x_t$ is a $I$-\textit{filter regular sequence} on $M$ if the following condition hold:
$$\mathrm{Supp}\big(
((x_1,\ldots,x_{i-1})M:x_i)/(x_1,\ldots,x_{i-1})M \big )\subseteq V(I)$$
for all $i = 1,\ldots,t$, where $V(I)$ denotes the set of prime
ideals containing $I$. This condition is equivalent to  $x_i \notin \fp$ for all $\fp \in \mathrm{Ass}_R M/(x_1, \ldots, x_{i-1})M \setminus V(I)$ and for all $i = 1, \ldots, t$. In the case $(R, \fm)$ is a local ring, we call an $\fm$-filter regular sequence of $M$ simply by {\it a filter regular sequence} of $M$.
\end{definition}
\begin{remark}\label{filter}
It should be noted that for any $t \ge 1$ we always can choose a $I$-filter regular sequence $x_1, \ldots, x_t$ on $M$. Indeed, by the prime avoidance lemma we can choose $x_1 \in I$ and $x_1 \notin \fp$ for all $\fp \in \mathrm{Ass}_RM \setminus V(I)$. For $i>1$ assume that we have $x_1, \ldots, x_{i-1}$, then we choose  $x_i \in I$ and $x_i \notin \fp$ for all $\fp \in \mathrm{Ass}_RM/(x_1, \ldots, x_{i-1})M \setminus V(I)$ by the prime avoidance lemma again. For more details, see \cite[Section 2]{AS03}.
\end{remark}
The $I$-filter regular sequence can be seen as a generalization of the well-known notion of regular sequences (cf.
\cite[Proposition 2.2]{NS95}).
\begin{lemma}\label{regular seq} A sequence $x_1, \ldots, x_t \in I$ is an $I$-filter regular sequence on $M$ if and only if for all $\frak p \in \mathrm{Supp}(M) \setminus V(I)$, and for all $i \le t$ such that $x_1, \ldots, x_i \in \frak p$ we have
$\frac{x_1}{1},\cdots,\frac{x_i}{1}$ is an $M_{\frak p}$-sequence.
\end{lemma}

\begin{corollary}\label{torsion} Let $x_1, \ldots, x_t \in I$ be an $I$-filter regular sequence on $M$. Then $H^i_{(x_1, \ldots, x_t)}(M)$ is $I$-torsion for all $i<t$.
\end{corollary}
\begin{proof}
For each $\frak p \in \mathrm{Supp}(M) \setminus V(I)$ we have either $(x_1, \ldots x_t)R_{\frak p} = R_{\frak p}$ or $x_1, \ldots, x_t$ is an $M_{\frak p}$-regular sequence by Lemma \ref{regular seq}. For the first case we have
$$(H^i_{(x_1, \ldots, x_t)}(M))_{\frak p} \cong H^i_{(x_1, \ldots, x_t)R_{\frak p}}(M_{\frak p}) = 0$$ for all $i \ge 0$.  For the second case we have
$$(H^i_{(x_1, \ldots, x_t)}(M))_{\frak p} \cong H^i_{(x_1, \ldots, x_t)R_{\frak p}}(M_{\frak p}) = 0$$ for all $i < t$ by the Grothendieck vanishing theorem \cite[Theorem 6.2.7]{BS98}. Therefore we have $(H^i_{(x_1, \ldots, x_t)}(M))_{\frak p} \cong 0$ for all $i < t$ and for all $\frak p \in \mathrm{Spec}(R) \setminus V(I)$. So $H^i_{(x_1, \ldots, x_t)}(M)$ is $I$-torsion for all $i<t$.
\end{proof}
It is well-known that local cohomology $H^i_{(x_1, \ldots, x_t)}(M)$ agrees with the $i$-th cohomology of the \v{C}ech complex with respect to the sequence $x_1, \ldots, x_t$
$$0 \to M \overset{d^0}{\longrightarrow} \bigoplus_i M_{x_i} \overset{d^1}{\longrightarrow} \bigoplus_{i<j} M_{x_ix_j}   \overset{d^2}{\longrightarrow} \cdots  \overset{d^{t-1}}{\longrightarrow} M_{x_1 \ldots x_t} \to 0. $$
The following simple fact plays an important role in our proof.
\begin{lemma}\label{vanish}
Let $x \in I$ be any element of $R$ and $M$ an $R$-module. Then $H^i_I(M_x) = 0$ for all $i \ge 0$.
\end{lemma}
\begin{proof} The multiplication map $M_x \overset{x}{\to} M_x$ is an isomorphism. It induces isomorphism maps $H^i_I(M_x) \overset{x}{\to} H^i_I(M_x)$ for all $i \ge 0$. But $H^i_I(M_x)$ is $I$-torsion, so it is $(x)$-torsion since $x \in I$. Therefore $H^i_I(M_x) = 0$ for all $i \ge 0$.
\end{proof}
By Corollary \ref{torsion} and Lemma \ref{vanish} the theorem of Nagel and Schenzel is a special case of the following theorem.
\begin{theorem}\label{general Thm} Let $I$ be an ideal of a Noetherian ring $R$ and $M$ an $R$-module. Suppose we have a complex of $R$-modules
$$0 \to M = M^0 \xrightarrow{d^0} M^1 \xrightarrow{d^1} M^2 \xrightarrow{d^2} \cdots \xrightarrow{d^{t-1}} M^t \to 0, \quad (\star) $$
where $H^j_I(M^i) = 0$ for all $i>0$ and for all $j \ge 0$. Suppose that the cohomology $H^i:= \mathrm{Ker}(d^i)/\mathrm{Im}(d^{i-1})$ is $I$-torsion for all $i <t$. Then we have the following isomorphism
\[
H^i_{I}(M) \cong
\begin{cases}
H^i&\text{ if } i<t\\
H^{i-t}_{I}(H^t) &\text{ if } i \ge t.
\end{cases}
\]
In particular, if $H^t$ is also $I$-torsion then $H^i_I(M) \cong H^i$ for all $i \ge 0$.
\end{theorem}
\begin{proof} It is sufficient to prove the first assertion. For all $j \ge 0$ we set $L^j: = \mathrm{Im}(d^{j-1})$ and $K^j := \mathrm{Ker}(d^j)$. Hence $H^j \cong K^j/L^j$ for all $j \ge 0$. We split the complex $(\star)$ into short exact sequences
\[
0 \to H^0 \to M \to L^1 \to 0  \tag{$A_0$}
\]
\[
0 \to L^1 \to K^1 \to H^1 \to 0 \tag{$B_1$}
\]
\[
0 \to K^1 \to M^1 \to L^2 \to 0  \tag{$A_1$}
\]
$$ \cdots $$
\[
0 \to L^j \to K^j \to H^j \to 0 \tag{$B_j$}
\]
\[
0 \to K^j \to M^j \to L^{j+1} \to 0 \tag{$A_j$}
\]
$$ \cdots $$
\[
0 \to L^{t-1} \to K^{t-1} \to H^{t-1} \to 0 \tag{$B_{t-1}$}
\]
\[
0 \to K^{t-1} \to M^{t-1} \to L^{t} \to 0 \tag{$A_{t-1}$}
\]
\[
0 \to L^t \to M^t \to H^t \to 0. \tag{$B_t$}
\]
Since $L^j$ and $K^j$ are submodules of $M^j$ for all $j \ge 1$, we have $H^0_I(L^j) \cong H^0_I(K^j) = 0$ for all $j \ge 1$. We also note that $H^j$ is $I$-torsion for all $j <t$ by the assumption, so $H^0_I(H^j) = H^j$ and $H^i_I(H^j) \cong 0$ for all $j<t$ and for all $i \ge 1$.\\
Now applying the functor $H^i_I(-)$ to the short exact sequence $(A_0)$ and using the above observations we have
$$H^0_I(M) \cong H^0$$
and
\[
H^i_I(M) \cong H^i_I(L^1) \tag{1}
\]
for all $i \ge 1$.\\
For each $j = 1, \ldots, t-1$, applying the local cohomology functor $H^i_I(-)$ to the short exact sequence $(A_j)$ we have $H^1_I(K^j) \cong 0$ and the isomorphism
\[
H^i_I(L^{j+1}) \cong H^{i+1}_I(K^j) \tag{$C_j$}
\]
for all $i \ge 1$. Furthermore, if we apply $H^i_I(-)$ for the short exact sequence $(B_j)$, then we obtain the short exact sequence
$$0 \to  H^0_I(H^j) \cong H^j \to H^1_I(L^j) \to H^1_I(K^j) \to 0,$$
and the isomorphism
\[
H^i_I(L^j) \cong H^i_I(K^j) \tag{$D_j$}
\]
for all $i \ge 2$. Note that $H^1_I(K^j) = 0$ as above, so
\[
H^j \cong H^1_I(L^j). \tag{2}
\]
By the isomorphisms $(C_j)$ and $(D_j)$ we have $H^i_I(L^{j+1}) \cong H^{i+1}_I(L^{j})$ for all $j = 1, \ldots, t-1$ and for all $i \ge 1$. We next show that $H^i_I(M) \cong H^i$ for all $i = 1, \ldots, t-1$. Indeed, using isomorphisms $(1), (2)$ and the above isomorphism we have
$$H^i_I(M) \overset{(1)}{\cong} H^i_I(L^1) \cong H^{i-1}_I(L^2) \cong \cdots \cong H^1_I(L^i) \overset{(2)}{\cong} H^i.$$
Therefore, we have showed the isomorphisms $H^i_I(M) \cong H^i$ for all $i = 0, \ldots, t-1$. Finally, for $i \ge t$ by similar arguments we have
$$H^i_I(M) \overset{(1)}{\cong} H^i_I(L^1) {\cong} H^{i-1}_I(L^2) {\cong} \cdots {\cong} H^{i-t+1}_I(L^t).$$
On the other hand, by applying the functor $H^i_I(-)$ to the short exact sequence $(B_t)$ we have
$$H^{i-t}_I(H^t) \cong  H^{i-t+1}_I(L^t)$$
for all $i \ge t$. Thus $H^i_I(M)  \cong H^{i-t}_I(H^t)$for all $i \ge t$, and we finish the proof.
\end{proof}
\begin{remark} Let $I = (x_1, \ldots, x_t)$ be an ideal of $R$. It is not hard to show that the cohomology of \v{C}ech complex $\breve{C}(x_1, \ldots, x_t;M)$ is always $I$-torsion for any $R$-module $M$. By the last assertion of Theorem \ref{general Thm} we obtain the well-known fact $H^i_I(M) \cong H^i(\breve{C}(x_1, \ldots, x_t;M))$ for all $i \ge 0$.
\end{remark}

\section{On the Frobenius test exponent for parameter ideals}
 In this section, let $R$ be a Noetherian ring containing a field of characteristic $p>0$. Let $F:R \to R, x \mapsto x^p$ denote the Frobenius endomorphism. If we want to notationally distinguish the source and target of the $e$-th Frobenius endomorphism $F^e: R \xrightarrow{x \mapsto x^{p^e}} R$, we will use $F_*^e(R)$ to denote the target. $F_*^e(R)$ is an $R$-bimodule, which is the same as $R$ as
an abelian group and as a right $R$-module, that acquires its left $R$-module structure via the $e$-th Frobenius
endomorphism $F^e$. By definition the $e$-th Frobenius endomorphism $F^e: R \to F_*^e(R)$ sending $x$ to $F_*^e(x^{p^e}) = x \cdot F_*^e(1)$ is an $R$-homomorphism.
\begin{definition}[\cite{H96}] Let $I$ be an ideal of $R$ we define
\begin{enumerate}
\item The {\it $e$-th Frobenius power} of $I$ is $I^{[p^e]} = (x^{p^e} \mid x \in I)$.
  \item The {\it Frobenius closure} of $I$, $I^F = \{x \mid  x^{p^e} \in I^{[p^e]} \text{ for some } e \ge 0\}$.
\end{enumerate}
\end{definition}
\begin{remark} An element $x \in I^F$ if it is contained in the kernel of the composition
 $$R \to R/I \cong R/I \otimes_R R \xrightarrow{\mathrm{id} \otimes F^e} R/I \otimes_R F_*^e(R)$$
 for some $e \ge 0$. Moreover $R$ is Noetherian, so $I^F$ is finitely generated. Therefore there exists an integer $e_0$ such that
 $$I^F = \mathrm{Ker}(R \to R/I \cong R/I \otimes_R R \xrightarrow{\mathrm{id} \otimes F^e} R/I \otimes_R F_*^e(R))$$
for all $e \ge e_0$.
\end{remark}
By the above discussion for every ideal $I$ there is an integer $e$ (depending on $I$) such that $(I^F)^{[p^e]} = I^{[p^e]}$. A problem of Katzman and Sharp \cite[Introduction]{KS06} asks in its strongest form: does there exist a number $e$, depending only on the ring $R$, such that, for every ideal $I$ we have $(I^F)^{[p^e]} = I^{[p^e]}$. A positive answer to
this question, together with the actual knowledge of a bound for $e$, would give an algorithm to compute the Frobenius closure $I^F$. We call such a number $e$ a {\it Frobenius test exponent} for the ring $R$. Unfortunately, Brenner \cite{B06} gave two-dimensional normal standard graded domains with no Frobenius test exponent. In contrast, Katzman and Sharp showed the existence of Frobenius test exponent if we restrict to class of parameter ideals in a Cohen-Macaulay ring. It leads the following question.
\begin{question}\label{Katzman Sharp} Let $(R, \fm)$ be an (equidimensional) local ring of prime characteristic $p$. Then does there exist an integer $e$ such that for every parameter ideal $\fq$ of $R$ we have $(\fq^F)^{[p^e]} = \fq^{[p^e]}$?
\end{question}
We definite the {\it Frobenius test exponent for parameter ideals} of $R$, $Fte(R)$, the smallest integer $e$ satisfying the above condition and $Fte(R) = \infty$ if we have no such $e$. Question \ref{Katzman Sharp} has affirmative answers when $R$ is either generalized Cohen-Macaulay by \cite{HKSY06} or $F$-nilpotent by \cite{Q18}, and is open in general. The Frobenius test exponent for parameter ideals is closely related with an invariant defined in terms of Frobenius action on local cohomology. For any ideal $I = (x_1, \ldots, x_t)$, the Frobenius endomorphism $F:R \to R$ and its localizations induce a natural Frobenius action on local cohomology $F:H^i_I(R) \to H^i_{I^{[p]}}(R) \cong H^i_{I}(R)$ for all $i \ge 0$. There is a very useful way of describing the top local cohomology. It can be given as the direct limit of Koszul cohomologies
$$
H^t_I(R) \cong \lim_{\longrightarrow} R/(x_1^n, \ldots, x_t^n),
$$
with the map in the system $\varphi_{n, m}: R/(x_1^n, \ldots, x_t^n) \to R/(x_1^m, \ldots, x_t^m)$ is multiplication by $(x_1 \ldots x_t)^{m - n}$ for all $m \ge n$. Then for each $\overline{a} \in H^t_{I}(R)$, which is the canonical image of $a+(x_1^n, \ldots, x_t^n)$, we find that $F(\overline{a})$ is the canonical image of $a^p +(x_1^{pn}, \ldots, x_t^{pn})$.

Notice that $H^i_{\fm}(R)$ is always Artinian for all $i \ge 0$. Let $A$ be an Artinian $R$-module with a Frobenius action $F: A \to A$. Then we define the {\it Frobenius closure} $0^F_A$ of the zero submodule of $A$ is the submodule of $A$ consisting all elements $z$ such that $F^e(z) = 0$ for some $e \ge 0$. $0^F_A$ is the nilpotent part of $A$ by the Frobenius action. By \cite[Proposition 1.11]{HS77} and \cite[Proposition 4.4]{L97} there exists a non-negative integer $e$ such that $0^F_A = \ker (A \overset{F^e}{\longrightarrow} A)$ (see also \cite{Sh07}). The smallest of such integers is called the {\it Hartshorne-Speiser-Lyubeznik number} of $A$ and denoted by $HSL(A)$. We define the {\it Hartshorne-Speiser-Lyubeznik number} of a local ring $(R, \frak m)$ as follows
$$HSL(R): = \min \{ e \mid   0^F_{H^i_{\fm}(R)} =   \ker (H^i_{\fm}(R) \overset{F^e}{\longrightarrow} H^i_{\fm}(R)) \text{ for all } i = 0, \ldots, d\}.$$
If $R$ is Cohen-Macaulay, then Katzman and Sharp \cite{KS06} showed that $Fte(R)$ is just $HSL(R)$. In this paper we will show that $Fte(R) \ge HLS(R)$ for any local ring $R$. We need the following result.
\begin{proposition}\label{limit Frobenius} Let $x_1, \ldots, x_t$ be a sequence of elements in $R$. Then we have
$$0^F_{H^t_{(\underline{x})}(R)} \cong  \varinjlim_n \frac{(x_1^n, \ldots, x_t^n)^F}{(x_1^n, \ldots, x_t^n)}.$$
\end{proposition}
\begin{proof} For each $e \ge 0$ the Frobenius action $F^e$ on $H^t_{(\underline{x})}(R)$ is the direct limit of the following commutative diagram:
$$
\begin{CD}
R/(\underline{x}) @>\varphi_{1,2}>>  R/(\underline{x}^{[2]})  @>\varphi_{2,3}>> R/(\underline{x}^{[3]}) @>>> \cdots \\
@VF^eVV @VF^eVV @VF^eVV (\star \star)\\
R/(\underline{x}^{[p^e]}) @>\varphi_{p^e,2p^e}>> R/(\underline{x}^{[2p^e]}) @>\varphi_{2p^e,3p^e}>> R/(\underline{x}^{[3p^e]}) @>>> \cdots
\end{CD}
$$
where each vertical map is the Frobenius homomorphism. For each $a \in (x_1^n, \ldots, x_t^n)^F$, it is clear that $a + (x_1^n, \ldots, x_t^n)$ maps to an element in $0^F_{H^t_{(\underline{x})}(R)}$. Thus we have an injection
$$ \varinjlim_n \frac{(x_1^n, \ldots, x_t^n)^F}{(x_1^n, \ldots, x_t^n)} \hookrightarrow 0^F_{H^t_{(\underline{x})}(R)}.$$
For the surjection, let $\overline{a}$ be any element of $0^F_{H^t_{(\underline{x})}(R)}$. By the system $(\star \star)$ there is an element $a \in R$ and an integer $n_1$ such that $a + (x_1^{n_1}, \ldots, x_t^{n_1})$ maps to $\overline{a}$. Let $e$ be an integer such that $F^e(\overline{a}) = 0$. Hence the image of $a^{p^e} + (x_1^{n_1p^e}, \ldots, x_t^{n_1p^e})$ is the zero in the limit. We can choose an integer $n_2 > n_1$ such that
$$\varphi_{n_1p^e, n_2p^e}(a^{p^e} + (x_1^{n_1p^e}, \ldots, x_t^{n_1p^e})) = 0 \in R/(x_1^{n_2p^e}, \ldots, x_t^{n_2p^e}).$$
Using the commutative diagram $(\star \star)$ we have
$$F^e(\varphi_{n_1, n_2}(a + (x_1^{n_1}, \ldots, x_t^{n_1}))) = \varphi_{n_1p^e, n_2p^e}(F^e(a + (x_1^{n_1}, \ldots, x_t^{n_1}))) = 0.$$
Therefore $\varphi_{n_1, n_2}(a + (x_1^{n_1}, \ldots, x_t^{n_1})) \in (x_1^{n_2}, \ldots, x_t^{n_2})^F/(x_1^{n_2}, \ldots, x_t^{n_2})$. Moreover this element maps to $\overline{a}$. This completes the proof.
\end{proof}
\begin{theorem}\label{inequality} Let $(R, \fm)$ be a local ring of positive characteristic $p>0$ of dimension $d$. Then $Fte(R) \ge HSL(R)$.

\end{theorem}
\begin{proof} There is nothing to do if $Fte(R) = \infty$. Therefore we can assume henceforth that $Fte(R) = e_0$ a finite number. By the prime avoidance theorem we can choose a system of parameters $x_1, \ldots, x_d$ of $R$ that is also a filter regular sequence. For all $t \le d$ and all $n \ge 1$ we have
\begin{eqnarray*}
((x_1^n, \ldots, x_t^n )^F)^{[p^{e_0}]} &\subseteq & \bigcap_{m \ge 1} ((x_1^n, \ldots, x_t^n, x_{t+1}^m, \ldots, x_d^m )^F)^{[p^{e_0}]} \\
 &= & \bigcap_{m \ge 1} (x_1^n, \ldots, x_t^n, x_{t+1}^m, \ldots, x_d^m )^{[p^{e_0}]}\\
 &= & (x_1^n, \ldots, x_t^n)^{[p^{e_0}]},
\end{eqnarray*}
where the first equation follows from the definition of Frobenius test exponent, and the second equation follows from Krull's intersection theorem. Hence
$$((x_1^n, \ldots, x_t^n )^F)^{[p^{e_0}]}  =  (x_1^n, \ldots, x_t^n)^{[p^{e_0}]}$$
for all $t \le d$ and for all $n \ge 1$. By Proposition \ref{limit Frobenius} we have
$$0^F_{H^t_{(x_1, \ldots, x_t)}(R)} \cong  \varinjlim_n\frac{(x_1^n, \ldots, x_t^n)^F}{(x_1^n, \ldots, x_t^n)}.$$
Following the above observation we have
$$\frac{(x_1^n, \ldots, x_t^n)^F}{(x_1^n, \ldots, x_t^n)} \xrightarrow{F^{e_0}} \frac{(x_1^{np^{e_0}}, \ldots, x_t^{np^{e_0}})^F}{(x_1^{np^{e_0}}, \ldots, x_t^{np^{e_0}})}$$
is the zero map for all $n \ge 1$ and for all $t \le d$, and so are the limit maps. Therefore $F^{e_0}(0^F_{H^t_{(x_1, \ldots, x_t)}(R)}) = 0$ for all $t \le d$. On the other hand by the Nagel-Schenzel theorem we have $H^t_{\fm}(R) \cong H^0_{\fm}(H^t_{(x_1, \ldots, x_t)}(R))$. Thus we can consider $H^t_{\fm}(R)$ as a submodule of $H^t_{(x_1, \ldots, x_t)}(R)$ that is compatible with Frobenius actions. Therefore $F^{e_0}(0^F_{H^t_{\fm}(R)}) = 0$ for all $t \le d$, that is $HSL(R) \le e_0$.
The proof is complete.
\end{proof}
We next show that the Frobenius test exponent has a good behavior under localization.
\begin{proposition}\label{localization}  Let $(R, \fm)$ be a local ring of positive characteristic $p>0$ of dimension $d$. Then $Fte(R) \ge Fte(R_{\fp})$ for all $\fp \in \mathrm{Spec}(R)$.
\end{proposition}
\begin{proof} We can assume that $Fte(R) = e_0$ a finite number. Let $t = \mathrm{ht}(\fp)$, and $I = (a_1, \ldots, a_t)R_{\fp}$ any parameter ideal of $R_{\fp}$. Following the proof of \cite[Proposition 6.9]{QS17} we can  choose a part of system of parameters $x_1, \ldots, x_t$ of $R$ such that $I = (x_1, \ldots, x_t)R_{\fp}$. Extending $x_1, \ldots, x_t$ to a full system of parameters $x_1, \ldots, x_d$ of $R$. We have
\begin{eqnarray*}
((x_1, \ldots, x_t )^F)^{[p^{e_0}]} &\subseteq & \bigcap_{m \ge 1} ((x_1, \ldots, x_t, x_{t+1}^m, \ldots, x_d^m )^F)^{[p^{e_0}]} \\
 &= & \bigcap_{m \ge 1} (x_1, \ldots, x_t, x_{t+1}^m, \ldots, x_d^m )^{[p^{e_0}]}\\
 &= & (x_1, \ldots, x_t)^{[p^{e_0}]}.
\end{eqnarray*}
Thus $((x_1, \ldots, x_t )^F)^{[p^{e_0}]} = (x_1, \ldots, x_t)^{[p^{e_0}]}$. Since Frobenius closure commutes with localization (see \cite[Lemma 3.3]{QS17}) we have
\begin{eqnarray*}
(I^F)^{[p^{e_0}]} &=& (((x_1, \ldots, x_t ) R_{\fp})^F)^{[p^{e_0}]}\\
& =& ((x_1, \ldots, x_t )^F R_{\fp})^{[p^{e_0}]}\\
&=& ((x_1, \ldots, x_t )^F)^{[p^{e_0}]}R_{\fp}\\
&=& (x_1, \ldots, x_t )^{[p^{e_0}]}R_{\fp}\\
&=& I^{[p^{e_0}]}.
\end{eqnarray*}
Therefore $Fte(R_{\fp}) \le e_0$. The proof is complete.
\end{proof}
Recall the a function $f: X \to  \mathbb{R} \cup \{ \infty \}$, where $X$ is a topological space, is called {\it upper semi-continuous} if for any $t \in  \mathbb{R} \cup \{ \infty \}$ we have $\{x \mid f(x) < t\}$ is an open set of $X$. We close this note with the following natural question, see \cite{Mu13} for the upper semi-continuity of function $HSL: \mathrm{Spec}(R) \to \mathbb{R} \cup \{ \infty \}, \fp \mapsto HLS(R_{\fp})$.

\begin{question} Is the function $Fte: \mathrm{Spec}(R) \to \mathbb{R} \cup \{ \infty \}, \fp \mapsto Fte(R_{\fp}),$ upper semi-continuous?

\end{question}

\begin{acknowledgement} The authors are grateful to the referee for careful reading of the paper and valuable
suggestions and comments.
\end{acknowledgement}

\end{document}